\documentclass[11pt]{article}

\usepackage{amsthm,amsmath,amssymb}

\usepackage{graphicx}

\usepackage[colorlinks=true,citecolor=black,linkcolor=black,urlcolor=blue]{hyperref}

\usepackage{stmaryrd}
\usepackage{textcomp}
\usepackage{enumerate}
\usepackage{setspace}
\usepackage{amssymb}
\usepackage{amsthm}
\usepackage{amsmath}
\usepackage{graphicx}
\usepackage{extarrows}
\usepackage{marvosym}
\usepackage{empheq}
\usepackage{latexsym}
\usepackage{endnotes}
\usepackage{fontenc}

\newtheorem{theorem}{Theorem}
\newtheorem{lemma}[theorem]{Lemma}

\theoremstyle{definition}

\newtheorem*{remark*}{Remark}

\def\PP{{\mathbb P}}

\renewcommand{\geq}{\geqslant}

\renewcommand{\leq}{\leqslant}

\newcommand\eps{\varepsilon}

\begin{document}
\author{Shoham Letzter\thanks{Department of Pure Mathematics and Mathematical Statistics, 
	Wilberforce Road,
	CB3 0WB Cambridge, 
	UK; 
	e-mail: \texttt{s.letzter}@\texttt{dpmms.cam.ac.uk}.
	}
	\and
	Julian Sahasrabudhe\thanks{Department of Mathematics, 
	University of Memphis, 
	Memphis Tennessee, 
	USA;
	e-mail: \texttt{julian.sahasra}@\texttt{gmail.com}.
	}
}
	
\title{On Existentially Complete Triangle-free Graphs}
\maketitle

\newcommand{\Addresses}{{
\bigskip
\footnotesize
\textsc{ Julian Sahasrabudhe, Department of Mathematics, University of Memphis, Memphis Tennessee, USA}\par\nopagebreak
 \texttt{julian.sahasra@gmail.com}
}}

\begin{abstract}
For a positive integer $k$, we say that a graph is $k$-existentially complete if for every $0 \leq a \leq k$, and every tuple of distinct vertices $x_1,\ldots,x_a$, $y_1,\ldots,y_{k-a}$, there exists a vertex $z$ that is joined to all of the vertices $x_1,\ldots,x_a$ and none of the vertices $y_1,\ldots,y_{k-a}$.  While it is easy to show that the binomial random graph $G_{n,1/2}$ satisfies this property with high probability for $k \sim c\log n$, little is known about the ``triangle-free'' version of this problem; does there exist a finite triangle-free graph $G$ with a similar ``extension property''. This question was first raised by Cherlin in 1993 and remains open even in the case $k=4$. 

We show that there are no $k$-existentially complete triangle-free graphs with $k >\frac{8\log n}{\log\log n}$, thus giving the first non-trivial, non-existence result on this ``old chestnut'' of Cherlin. We believe that this result breaks through a natural barrier in our understanding of the problem. 
\end{abstract}

\section{Introduction}

We define an \emph{embedding} of a graph $H$ \emph{into} a graph $G$ to be a graph isomorphism from $H$ to an induced subgraph of $G$.\footnote{ Often this is called an \emph{isomorphic embedding}, but as all embeddings that we consider will be isomorphic embeddings, we drop the extra qualifier.} Let a \emph{partial embedding} of a graph $H$ \emph{into} $G$ be a graph isomorphism between an induced subgraph of $H$ and an induced subgraph of $G$. Say that an embedding $\tilde{\theta} :H \rightarrow G$ \emph{extends} a partial embedding $\theta$ of $H$ into $G$ if $\tilde{\theta}$ is identically $\theta$, when $\tilde{\theta}$ is restricted to the domain of $\theta$. 

The Rado graph $R$, an interesting and important object in its own right \cite{camRand,EHP,OrientedRado,Rado}, is the unique countable graph, up to isomorphism, with the property that every partial embedding of a countable graph $G$ may be extended to a complete embedding of $G$. Interestingly, there exist many examples of finite graphs that ``approximate'' the Rado graph in this sense. For example, the binomial random graph $G_{n,1/2}$ has the property that every partial embedding of a graph on $c\log n$ vertices extends to a complete embedding. 

It is not hard to see that there is a natural analogue of the Rado graph for the class of triangle-free graphs. That is, there is a \emph{countable}, triangle-free graph $G$ so that every partial embedding of an arbitrary, countable, triangle-free graph $H$ extends to a complete embedding of $H$. The 
question of whether there exist \emph{finite} graphs with a similar properly was raised and studied by Cherlin \cite{Cher1, Cher2} in the context of logic and model theory and has recently made its way over to combinatorics by way of Even-Zohar and Linial \cite{Linial}. This is the question that we pursue in this paper. 

To be more precise, let us fix some terms. We call a graph \emph{$k$-existentially complete} if every partial embedding of a graph on $k+1$ vertices extends to a complete embedding. We call a graph $k$-\emph{existentially complete triangle-free} (Henceforth $k$-ECTF) if $G$ is triangle-free and every partial embedding of a triangle-free, $(k+1)$-vertex graph extends to a complete embedding. Equivalently, a graph is $k$-ECTF if for every $0 \leq a \leq k$ and every tuple of distinct vertices $x_1,\ldots,x_a$, $y_1,\ldots,y_{k-a}$ there exists a vertex adjacent to all of $x_1,\ldots,x_a$ and none of $y_1,\ldots,y_{k-a}$, \emph{provided} $x_1,\ldots,x_a$ form an independent set.  

In 1993, Cherlin asked \cite{Cher1} if there exist finite $k$-ECTF graphs for every fixed $k \in \mathbb{N}$. To date this problem remains poorly understood \cite{Cher2} and the state-of-the-art can be summarized as follows. The case $k=1$ is trivial; a graph is $2$-ECTF if and only if it is maximal triangle-free, twin-free and not a $C_5$ or a single edge; there are various (non-trivial) constructions for $3$-ECTF graphs \cite{Pach,Cher1,Cher2,Linial}; and the case $k = 4$ is open. 

Our belief is along the lines of Even-Zohar and Linial, who have conjectured that no such graphs exist for $k \geq k_0$, where $k_0 \in \mathbb{N}$. In the present paper we take a step towards this conjecture by giving a non-trivial restriction on the maximum possible value of $k$, relative to $n$. Let $f(n)$ be the largest integer $k$ for which there exists a $k$-ECTF graph on $n$ vertices. While an easy argument gives that $f(n) \leq c\log n $, our main result gives an asymptotic improvement over this estimate, thereby giving a first non-trivial restriction on $f(n)$.

\begin{theorem} \label{thm:main}
Let $n \in \mathbb{N}$ be sufficiently large. There do not exist $k$-ECTF graphs on $n$ vertices, with $k > \frac{8\log n }{\log \log n}$. That is, $f(n) = O\left(\frac{\log n}{\log\log n}\right)$. 
\end{theorem}

One might interpret Theorem~\ref{thm:main} as giving the first concrete evidence that the triangle-free version of the problem is substantially different than the problem without the restriction on triangles. Indeed, with high probability, $G_{n,1/2}$ is $k$-existentially complete with $k \sim c\log n$ and thus essentially matches the trivial bound.

We point out that a closely related question on triangle-free graphs was raised and studied by Erd\H{o}s and Fajtlowicz \cite{EandFaj} and later by Pach \cite{Pach}; for
$k \in \mathbb{N}$, what is the smallest maximum degree in a triangle-free graph $G$ with the property that every independent set of size at most $k$ has a common neighbour.
Little is known about this question except in this case where $k$ is large:  Pach \cite{Pach} gave a classification of graphs where $k \geq n$, that is where \emph{all} independent sets have a common neighbour. This was later later strengthened for $k \geq c\log n$ by Erd\H{o}s and Pach \cite{EandP}.

\section{Proof of Main Theorem}
Given a finite set $X$, we say that $\mu$ is a \emph{probability measure on} $X$ if $\mu$ is a probability measure on the power set of $X$. Of course,
$\mu$ is determined by the values $\mu(\{x\})$ for $x \in X$. For a graph $G = (V(G),E(G))$, and disjoint subsets $X,Y \subseteq V(G)$, let $G[X,Y]$ denote the \emph{induced bipartite graph} on vertex set $X \cup Y$, with bipartition $\{X, Y\}$, and $x \in X $ joined to $y \in Y$ if and only if $xy \in E(G)$. 

Let $G$ be a bipartite graph with vertex partition $\{A, B\}$. For $s,t \in \mathbb{N}$, we say $G$ is $(s,t)$-\emph{separating for} $A$ if, for every pair of disjoint subsets $S,T \subseteq A$ with $|S| \leq s$ and $|T|\leq t$, there exists a vertex $v \in B$ so that $v$ is joined to all of $S$ and none of $T$. 
\paragraph{}
It is easy to see that if $k \in \mathbb{N}$ and $G = (A,B,E)$ is a bipartite graph which is $(k,k)$-separating for $A$, then $|B|\geq 2^k$. The following lemma, gives a strengthened bound when we impose a restriction on the neighbourhoods of vertices in $B$.

\begin{lemma} 
\label{Weightedtrivial}
For $k \in \mathbb{N}$, let $G$ be a bipartite graph with bipartition $\{A, B\}$ with $|A| \geq 2k$, and let $\mu$ be a probability measure on $A$. If $G$ is $(k,0)$-separating for $A$ and $\mu(N(x) ) < \varepsilon $ for each $x \in B$, then $|B| > 1/\varepsilon^k$.
\end{lemma}
\begin{proof}
Sample the points $x_1, \ldots x_k \in A $ independently at random and according to the distribution $\mu$. Then
\begin{align*}
 1 =&\,\,\, \mathbb{P}\left( x_1, \ldots, x_k \in N(x) \text{ for some } x \in B \right) \\
 \leq&\,\sum_{x \in B} \mathbb{P}\left( x_1, \ldots x_k \in N(x) \right) \\ 
 =&\,\sum_{x \in B} \mu(N(x))^k \leq |B|\varepsilon^k ,
\end{align*} 
thus completing the proof. 
\end{proof}
\paragraph{}
For, $s,t \in \mathbb{N}$, let $G = (A, B, E)$ be a bipartite graph that is $(s,t)$-separating for $A$. We now define a measure on $B$ that measures how well a given subset of $B$ covers the $s$-tuples of $A$. In particular, define the \emph{covering measure} $\mu_{G,s,A}$, with respect to $G$, by defining a way of sampling it: first sample $X_1,\ldots,X_s \in A$ independently and uniformly from $A$. Then, uniformly at random, choose a vertex among all vertices $v \in B$ so that $X_1,\ldots,X_s \in N(v)$. A key property of this measure is that for every $B' \subseteq B$, we have that
\begin{equation} \label{equ:BoundOnmu} 
\mu_{G,s,A}(B') \leq  \PP(X_1,\ldots,X_s \in N(x), \text{ for some } x \in B') .
\end{equation}  
Here $\mathbb{P}$ denotes the uniform measure on $A$ for the $X_1,\ldots,X_s$.
The following lemma says that if $G = (A,B,E)$ is $(s,0)$-separating for $A$ and a set $B' \subset B$ is given large mass by $\mu_{G,s,A}$, then the neighbourhoods of $x \in B'$ ``expand'' and collectively cover many vertices of $A$.  

\begin{lemma} \label{expansionLemma} 
For $k \in \mathbb{N}$, let $G = (A, B , E)$ be a bipartite graph which is $(k,0)$-separating for $A$ and let $\mu = \mu_{G,k,A}$ be the covering measure defined on $B$. If $B' \subseteq B$ has $\mu(B') > \eps$ for some $\eps >0$, then 
\[ \left| \bigcup_{x \in B'} N(x) \right| \geq \left(1 - \frac{1}{k}\log \left(\eps^{-1} \right)\right) |A| .\]
\end{lemma}
\begin{proof}
Write $\left| \bigcup_{x \in B'} N(x) \right| = (1 - \eta)|A|$ for some $0 < \eta < 1$. Then if $X_1,\ldots,X_k$ are sampled independently and uniformly from $A$, we have
\begin{align} \label{equ:inExpLemma}
\begin{split}
	& \mathbb{P}(X_1, \ldots ,X_k \in N(x) \text{ for some } x \in B') \\
	\leq\, \,& \mathbb{P}\big( X_1, \ldots, X_k \in \bigcup_{ x \in B'} N(x)\big) \\
	\leq\, \,& (1-\eta)^k \leq e^{-k \eta}. 
\end{split}
\end{align} 
Now apply the observation at~(\ref{equ:BoundOnmu}) to (\ref{equ:inExpLemma}) to obtain the inequality
\[ \eps < \mu(B') \leq \mathbb{P}\left( X_1, \ldots ,X_k \in N(x) \text{ for some } x \in B' \right) \leq e^{-k \eta} . \]
Taking logarithms gives $\eta < \frac{1}{k}\log\left( \varepsilon^{-1} \right)$, as desired.\end{proof}

We also require a basic fact about triangle-free graphs. 

\begin{lemma} A triangle-free graph on $n$ vertices contains an independent set of size $\geq \lfloor \sqrt{n} \rfloor$
\end{lemma}
\begin{proof} If $G$ contains a vertex of degree at least $\lfloor \sqrt{n} \rfloor$ then the neighbourhood of this vertex is an independent set and we are done. Otherwise, all neighbourhoods are of size at most $\lfloor \sqrt{n} \rfloor - 1$. In this latter case we may greedily construct a proper colouring of $G$ with at most $\sqrt{n}$ colours. There will be a colour-class of size at least $\sqrt{n}$.
\end{proof}
We are now in a position to give the proof of our main theorem. 
\begin{proof}[Proof of Theorem~\ref{thm:main}]
Suppose that $G$ is a $2k$-ECTF graph on $n$ vertices with $k \geq  \frac{4\log n }{\log \log n }$. To reduce clutter, let $\eps = 4(\log\log n)^{-1}$ so that $k \geq  \eps \log n $. Fix an independent set $I\subseteq V(G)$ with $|I| \geq \lfloor \sqrt{n} \rfloor$ and choose $x_0 \in I$. Then set $J = I \setminus \{x_0 \}$. We define a procedure that will discover a collection of more than $n$ distinct vertices in $G$, thus giving a contradiction. Let us set
$\alpha =  \frac{3}{k}\log (\eps^{-2}) $ and note the inequalities
\begin{equation} \label{equ:BoundOnAlpha} \alpha^{-k/2} > n , 
\end{equation} and
\begin{equation} \label{equ:LowerboundAlpha} \frac{2}{k} \log (\eps^{-2}) + \frac{k}{\sqrt{n} - 2} \leq \frac{3}{k}\log (\eps^{-2}) = \alpha, 
\end{equation} which hold for $n$ sufficiently large. 

\paragraph{}
We prove the following statement by induction on $t \in [0,n+1]$: for each $t \in [0,n+1]$ we may find a vertex $ w_t \in V(G)$, and a set $L_t \subseteq J^{\frac{k}{2}}$ so that the following conditions hold.
\begin{enumerate}
\item  The vertices $w_1,\ldots,w_t$ are distinct.
\item  If $(v_1,\ldots,v_{k/2}) \in L_t$, then $v_1,\ldots,v_{k/2}$ are not all contained in any of the neighbourhoods $\{ N(w_i) \}_{i=1}^t$. That is, 
\[ (v_1,\ldots,v_{k/2}) \not\in \bigcup_{i=1}^t (N(w_i)) ^{k/2}. \]
\item  We have $|L_t| \geq \left( 1 - t \alpha^{k/2} \right)|J|^{k/2}$.
\end{enumerate}
For the basis step, set $L_0 = J^{k/2}$. Now assume that we have defined distinct vertices $w_1,\ldots,w_{t-1}$ and a set $L_{t-1}$ satisfying the above. 
We show that we may find appropriate $w_t$ and $L_t$.
\paragraph{}
Note that $|L_{t-1}| \geq 1$, as $|L_{t-1}| \geq |J|^{k/2}(1 - (t-1)\alpha^{k/2}) \geq |J|^{k/2}(1 - n\alpha^{k/2})>0$, as $\alpha^{-k/2} > n$, by the inequality at (\ref{equ:BoundOnAlpha}). So we may fix $y_1, \ldots , y_{k/2} \in J$ so that $(y_1, \ldots ,y_{k/2}) \in L_{t-1}$. Define $B \subseteq V(G)$ to be the collection of vertices in $G$ that are adjacent to $x_0$ and not adjacent to any of $y_1,\ldots,y_{k/2}$. Note that since each vertex in $B$ joins to $x_0$, $B$ is an independent set. Now put $A = I \setminus \{ x_0, y_1,\ldots,y_{k/2} \}$ and consider $G[A,B]$. Observe that $G[A,B]$ is $(k/2,k/2-1)$-separating for $A$ and therefore $|B| \geq 2^{k/2-1}$, by the trivial bound. Let $\mu = \mu_{G[A,B],k/2,A}$ be the covering measure defined on $B$, with respect to the bipartite graph $G[A,B]$.
\paragraph{}
Define $Y$ to be the set of vertices in $G$ that \emph{are} joined to \emph{all} of $y_1,\ldots,y_{k/2}$. 
Note that the graph $G[B,Y]$ is $(k/2,k/2)$-separating for 
$B$, as there are no edges between $y_1,\ldots,y_{k/2}$ and $B$ and $B$ is an independent set in $G$. We now claim that there exists a vertex $w \in Y$ with $\mu(N_B(w)) > \varepsilon^2$. Suppose to the contrary that $\mu(N_B(x)) < \eps^2 $ for all $x \in Y$. Since $G[B,Y]$ is $(k/2,k/2)$-separating for $B$, we may apply Lemma~\ref{Weightedtrivial} to learn that $|Y| > \frac{1}{\varepsilon^k} = n$, which is a contradiction.\\
\paragraph{}
So we may choose some $w \in Y$ with  $\mu(N_B(w)) \geq \eps^2 $ and apply Lemma~\ref{expansionLemma} to learn that 
\begin{equation} \label{equ:expansion} \left| \bigcup_{x \in N_B(w)} N_A(x) \right| \geq \left(1 - \frac{2}{k}\log \left( \varepsilon^{-2} \right)\right)|A|.
\end{equation} The key here is that $w$ is not adjacent to any of the vertices in the union on the left hand side of (\ref{equ:expansion}), as this would create a triangle. 
Thus, (\ref{equ:expansion}) tells us that $w$ is adjacent to at most $ \frac{2|A|}{k}\log \left( \varepsilon^{-2} \right)$ vertices in $A$ and
thus $w$ is adjacent to at most $\frac{2|A|}{k}\log \left( \varepsilon^{-2} \right) + k/2 $ vertices in $J$.
As $|A| < |J|$, $w$ covers at most 
\begin{equation} \label{equ:BoundOnCoveredTuples} 
	|J|^{k/2}\left( \frac{2}{k} \log (\eps^{-2})  + \frac{k}{2|J|} \right) ^{k/2}  \leq  |J|^{k/2}\alpha^{k/2}
\end{equation} 
$k/2$-tuples in $J^{k/2}$. Here we have used the inequality $|J| = |I| - 1 \geq \lfloor \sqrt{n} \rfloor - 1$ and the inequality at (\ref{equ:LowerboundAlpha}). So we define $w_t = w$ and set
\[ L_t = L_{t-1} \setminus \left\lbrace (v_1,\ldots,v_{k/2}) : v_1,\ldots,v_{k/2} \in N(w) \cap J \right\rbrace.
\] By induction and the bound at (\ref{equ:BoundOnCoveredTuples}) we have $|L_t| \geq |J|^{\frac{k}{2}}\left( 1 - t\alpha^{k/2}\right)$. Finally, we note that $w_t$ must
be distinct from $w_1,\ldots,w_{t-1}$ as $w_t$ is joined to all of $y_1,\ldots,y_{k/2}$ which is not true of any of the $w_1,\ldots,w_{t-1}$, by the induction hypothesis. 
\paragraph{}
So, by induction, we have constructed $n+1$ distinct vertices in a $n$-vertex graph; a contradiction. This implies that there are no $l$-ECTF graphs with $l = 2k \geq \frac{8\log n}{\log\log n}$, thus completing the proof of Theorem~\ref{thm:main}. 
\end{proof}
\section{Acknowledgements}
We should like to thank B\'{e}la Bollob\'{a}s for introducing us to the problem of Cherlin. The second named author would like to thank Trinity College Cambridge for support. He would also like to thank the Cambridge Combinatorics group for their hospitality.

\bibliographystyle{plain}
\bibliography{ECTFbib}

\end{document}